\documentclass[11pt]{article}

\usepackage[dvips]{epsfig} 
\usepackage{amssymb,amsmath,amsthm,mathrsfs} 
\usepackage{graphicx}
\usepackage{lineno}
\usepackage{multirow}

\usepackage{url}
\usepackage{enumerate}
\usepackage{colordvi,color,pspicture}
\usepackage{xcolor}
\usepackage{psfrag}
\usepackage{pgf,tikz}
\usepackage{mathrsfs}
\usetikzlibrary{arrows}
\usetikzlibrary{decorations.markings}
\usepackage{authblk}

\newtheorem{theorem}{Theorem}[section]
\newtheorem{corollary}[theorem]{Corollary}
\newtheorem{lemma}[theorem]{Lemma}
\newtheorem{proposition}[theorem]{Proposition}

\newtheorem{remark}[theorem]{Remark}

\theoremstyle{definition}

\begin{document}
\title{Total $k$-Uniform Graphs}
\author[1]{Selim Bahad{\i}r \thanks{sbahadir@ybu.edu.tr}}
\author[2]{Didem G\"{o}z\"{u}pek \thanks{didem.gozupek@gtu.edu.tr}}
\author[3]{Oğuz Doğan,\thanks{odogan13@ku.edu.tr} }
\affil[1]{Department of Mathematics,
	Ankara Y\i ld\i r\i m Beyaz\i t University, Turkey}
\affil[2]{Department of Computer Engineering,
	Gebze Technical University, Turkey}
\affil[3]{Department of Computer Engineering,
	Gebze Technical University, Turkey}
\date{\today}

\maketitle
\begin{abstract}
A sequence of vertices in a graph $G$ without isolated vertices is called a total dominating sequence if every vertex $v$ in the sequence has a neighbor which is adjacent to no vertex preceding $v$ in the sequence, and at the end every vertex of $G$ has at least one neighbor in the sequence.
Minimum and maximum lengths of a total dominating sequence is the total domination number of $G$ (denoted by $\gamma_t(G)$) and the Grundy total domination
number of $G$ (denoted by $\gamma_{gr}^t(G)$), respectively.
In this paper, we study graphs with equal total domination number and Grundy total domination number.
For every positive integer $k$,
we call $G$ a total $k$-uniform graph if $\gamma_t(G)=\gamma_{gr}^t(G)=k$.
We prove that there is no total $k$-uniform graph when $k$ is odd.
In addition, we present a total 4-uniform graph which stands as a counterexample for a conjecture by \cite{gologranc2019graphs} and 
provide a connected total 8-uniform graph.
Moreover, we prove that every total $k$-uniform, connected and false twin-free graph is regular for every even $k$.
We also show that there is no total $k$-uniform chordal connected graph with $k\geq 4$ and characterize all total $k$-uniform chordal graphs.
\end{abstract}


\section{Introduction}\label{sec:intro}

Let $G$ be a simple graph with vertex set $V(G)$ and edge set $E(G)$.
The neighborhood of a vertex $v\in V(G)$, denoted by $N(v)$, is the set of vertices adjacent to $v$.
The closed neighborhood of a vertex $v\subseteq V(G)$, denoted by $N[v]$, is $N(v)\cup \{v\}$.
A subset $A$ of $V(G)$ is called a \emph{dominating set} of $G$ if every vertex in $V(G) \backslash A$ has at least one neighbor in $A$.
If $G$ has no isolated vertices, a subset $A\subseteq V(G)$ is called a \emph{total dominating set} of $G$ if every vertex of $V(G)$ is adjacent to at least one member of $A$.
The \emph{total domination number} of $G$ with no isolated vertices, denoted by $\gamma_t(G)$, is the minimum size of a total dominating set of $G$.

A sequence $S=(v_1, \dots , v_k)$ of distinct vertices of $G$ is a \emph{legal (open neighborhood) sequence} if
$$N(v_i)\backslash \bigcup_{j=1}^{i-1} N(v_j)\neq \emptyset$$
holds for every $i \in \{2, \dots , k\}$.
If, in addition, $\{v_1,\dots,v_k\}$ is a total dominating set of $G$, then we call $S$ a \emph{total dominating sequence} of G.
The maximum length of a total dominating sequence in $G$ is called the \emph{Grundy total domination number} of $G$ and it is denoted by $\gamma_{gr}^t(G)$.

It is clear that length of a total dominating sequence in $G$ is at least $\gamma_t(G)$, and any permutation of a minimum total dominating set of $G$ forms a total dominating sequence attaining this lower bound.
In this paper,
we study graphs in which the Grundy total domination number is equal to the total domination number.
A graph $G$ is called \emph{total $k$-uniform} if $\gamma_t(G)=\gamma_{gr}^t(G)=k$.
In other words, a graph $G$ is total $k$-uniform if and only if every total dominating sequence is of length $k$.
Total $k$-uniform graphs are indeed total domination version of \emph{$k$-uniform graphs} introduced in \cite{erey2020uniform}.
A sequence $(v_1, \dots , v_k)$ is called \emph{dominating closed neighborhood sequence} if $\{v_1,\dots,v_k\}$ is a dominating set and $N[v_i]\backslash \bigcup_{j=1}^{i-1} N[v_j]\neq \emptyset$
holds for every $i \in \{2, \dots , k\}$.
A graph is called $k$-uniform whenever every dominating closed neighborhood sequence is of length $k$.
$k$-uniform graphs with $k\leq 3$ are characterized in \cite{brevsar2014dominating}, whereas the work in \cite{erey2020uniform} provided the complete characterization of $k$-uniform graphs.

Numerous variants of Grundy total domination such as Grundy domination, Z-Grundy domination and L-Grundy domination exist in the literature \cite{brevsar2014dominating, brevsar2016atomic, brevsar2016toroidal, brevsar2017zeroforcing}. The work in \cite{brevsar2017zeroforcing} investigated the relations between these types of Grundy domination as well as the relation between the Z-Grundy domination number and the zero forcing number of a graph.

Total domination concept in graphs was introduced in 1980 \cite{cockayne1980total} and has been studied extensively in the literature (see \cite{henning2013total}). The parameter $\gamma_{gr}^t(G)$ was first introduced by \cite{brevsar2016total}, who obtained bounds on $\gamma_{gr}^t(G)$ for trees and regular graphs in terms of other graph variants. The decision versions of both the total domination number and the Grundy total domination number are NP-complete \cite{brevsar2016total}. Indeed, the problem of finding the Grundy total domination number is NP-hard in bipartite graphs \cite{brevsar2016total} and split graphs \cite{brevsar2018total}, while it is solvable in polynomial time in trees, bipartite distance-hereditary graphs, and $P_4$-tidy graphs \cite{brevsar2018total}.

Various bounds for Grundy total domination number have been obtained for several graph classes such as regular graphs and graph products \cite{brevsar2017grundy, brevsar2016total}. The work in \cite{brevsar2016total} also characterized the graphs where the Grundy total domination number attains its trivial upper bound $|V(G)|$.
In their work, it is additionally shown that complete multipartite graphs are the only total 2-uniform graphs and there are no total 3-uniform graphs.
In this paper, we generalize their result for total 3-uniform graphs and prove that there does not exist any total $k$-uniform graph when $k$ is odd and hence, we partially solve the open problem (characterizing the graphs $G$ such that $\gamma_t(G)=\gamma_{gr}^t(G)=k$ for $k\geq 4$) posed by \cite{brevsar2016total}.
We also obtain that removing all vertices with the same neighborhood but except one from a connected total $k$-uniform graph gives rise to a regular graph.

\cite{gologranc2019graphs} characterized total 4-uniform bipartite graphs, showed that there is no connected total 4-uniform chordal graph and established a correspondence between regular total 6-uniform bipartite graphs and some certain finite projective planes.
The authors also claimed that any connected total 4-uniform graph is bipartite.
In this paper, we disprove their conjecture by exhibiting a graph with fifteen vertices.
We additionally prove that there are no connected total $k$-uniform chordal graphs when $k\geq 4$; hence, we classify all total $k$-uniform chordal graphs.

The remainder of this paper is organized as follows.
In Section \ref{sec:totalkodd} we provide a reduction from total $k$-uniform graphs to total $(k-2)$-uniform graphs, which is essential to this paper and implies that there does not exist a total $k$-uniform graph when $k$ is odd.
In Section \ref{sec:newtkunif}, 
we present a connected, non-bipartite and total 4-uniform graph together with a connected total 8-uniform graph.
In Section \ref{sec:reg}, we show that connected, false twin-free and total $k$-uniform graphs are regular.
Section \ref{sec:chordal} contains the complete characterization of total $k$-uniform chordal graphs.
Discussion and conclusions are provided in Section \ref{sec:dis}.

\section{A Reduction from Total k-Uniform Graphs to Total (k-2)-Uniform Graphs}\label{sec:totalkodd}

If $G$ is a graph with no isolated vertices, then every legal sequence in $G$ which is not a total dominating sequence can be extended to a total dominating sequence of $G$. 
Moreover, a graph $G$ contains a total dominating sequence if and only if $G$ has no isolated vertices. 
We shall implicitly make use of these observations in the sequel. 

A graph whose all minimal total dominating sets are of the same size is called a \emph{well-totally-dominated} graph (see \cite{wtd}).
Well-totally-dominated graphs are initially introduced and studied in \cite{HaRa97}.
Notice that any permutation of a minimal total dominating set generates a total dominating sequence and therefore, every total $k$-uniform graph is a well-totally-dominated graph; that is, the family of total $k$-uniform graphs is a subset of well-totally-dominated graphs.
It is shown in \cite{HaRa97} that removing two adjacent vertices together with their neighbors from a well-totally-dominated graph gives rise to another well-totally-dominated graph as long as the resulting graph has no isolated vertex.
Based on this idea,
we show some lemmas which are essential for the main results in this paper.

\begin{lemma}\label{lem:kunifiso}
If $v_1v_2$ is an edge of a total $k$-uniform graph $G$ where $k\geq 2$, then $G\backslash (N[v_1]\cup N[v_2]) $ has no isolated vertices.
\end{lemma}
\begin{proof}
Suppose that $G\backslash (N[v_1]\cup N[v_2])$ has an isolated vertex $v_0$.
Notice that $v_0$ is adjacent to none of $v_1$ and $v_2$, but every neighbor of $v_0$ is adjacent to $v_1$ or $v_2$.
In other words, $v_0\notin N[v_1]\cup N[v_2]$ and $N(v_0)\subseteq N(v_1)\cup N(v_2)$.

Since $v_1\in N(v_2)\backslash N(v_1) $, $(v_1,v_2)$ is a legal sequence and can be extended to a total dominating sequence of $G$, say $(v_1,v_2,\dots,v_k)$.
Note that $v_i\neq v_0$ for every $3\leq i \leq k$ as $N(v_0)\subseteq N(v_1)\cup N(v_2)$.

Now consider the sequence $(v_0,v_1,v_2,\dots,v_k)$.
Since $v_2 \in N(v_1)\backslash N(v_0)$, $v_1 \in N(v_2) \backslash (N(v_0)\cup N(v_1))$ and 
$$ N(v_i) \backslash \cup_{j=0}^{i-1} N(v_j)=N(v_i) \backslash \cup_{j=1}^{i-1} N(v_j)\neq \emptyset$$
holds for every $3\leq i\leq k$,
$(v_0,v_1,v_2,\dots,v_k)$ is a legal sequence.
Moreover, the set $\{v_0,\dots,v_k\}$ is also a total dominating set as it contains a total dominating set $\{v_1,\dots,v_k\}$.
Therefore, $(v_0,\dots,v_k)$ is a total dominating sequence of length $k+1$, contradiction.
\end{proof}

\begin{lemma}\label{lem:reduction}
Let $G$ be a total $k$-uniform graph with no isolated vertices where $k\geq 3$.
If $uv\in E(G)$, then $G\backslash (N[u]\cup N[v]) $ is a total $(k-2)$-uniform graph with no isolated vertex.
\end{lemma}
\begin{proof}
First note that $\{u,v\}$ is not a total dominating set of $G$ since $k\geq 3$, and hence $G\backslash (N[u]\cup N[v]) $ is not the empty graph.
Moreover, by Lemma \ref{lem:kunifiso} we have that $G\backslash (N[u]\cup N[v]) $ has no isolated vertex.
Let $(v_1,\dots,v_m)$ be any total dominating sequence of $G\backslash (N[u]\cup N[v]) $.
It is easy to verify that $(v_1,\dots,v_m,u,v)$ is a total dominating sequence of $G$.
Thus, we get $m+2=k$, that is, $m=k-2$, and the result follows.
\end{proof}

Now, for odd $k\geq 3$, applying Lemma \ref{lem:reduction} $(k-1)/2$ times beginning with a total $k$-uniform graph yields a total 1-uniform graph, which does not exist since $\gamma_t(G)\geq 2$ for every $G$ with no isolated vertex.
Consequently, we obtain the following result:
\begin{theorem}\label{thm:noodd}
There does not exist a total $k$-uniform graph where $k$ is an odd positive integer.
\end{theorem}

Theorem 3.2 in \cite{brevsar2014dominating} implies that for every $l$ with $\gamma_t(G)\leq l \leq \gamma_{gr}^t(G)$ there exists a total dominating sequence of length $l$ in $G$ (which is also Corollary 8.2 in \cite{brevsar2016total}).
Combining this fact with Theorem \ref{thm:noodd} we obtain the following result:
\begin{corollary}\label{cor:eventotaldomseq}
Every graph with no isolated vertex has a total dominating sequence of even length.
\end{corollary}
\begin{remark}
Note that Corollary \ref{cor:eventotaldomseq} can also be proven by using the ideas in the proofs of Lemma \ref{lem:kunifiso} and Lemma \ref{lem:reduction}.
\end{remark}

Two distinct vertices $u$ and $v$ of a graph $G$ are called \emph{false twins} if $N(u) = N(v)$.
A graph is \emph{false twin-free} if it has no false twins. 
Now notice that removing one of the false twins or creating a false twin of a vertex changes neither the total domination number nor the Grundy total domination number.
Therefore, the question of characterizing total $k$-uniform graphs is only interesting for false twin-free graphs.

Let $\mathcal{G}_k$ be the family of total $k$-uniform and false twin-free graphs for every even $k$. By Theorem 4.4 in \cite{brevsar2016total} we see that the family of total 2-uniform graphs is the set of all complete multipartite graphs (including complete graphs).
Therefore, $\mathcal{G}_2$ is the set of complete graphs with at least two vertices.

\begin{lemma}\label{lem:falsetwinreduction}
Let $k\geq 4$ be an even integer and $G\in \mathcal{G}_k$.
Then for any edge $uv\in E(G)$ the graph $G\backslash (N[u]\cup N[v]) $ is also false twin-free.
\end{lemma}
\begin{proof}
Let $H=G\backslash (N[u]\cup N[v]) $.
To the contrary, assume that $H$ has some false twins, say $w$ and $w'$.
Since $w$ and $w'$ are not false twins in $G$, there exists a vertex $x$ in $V(G)\backslash \{w,w'\}$ adjacent to exactly one of them.
Without loss  of generality suppose that $xw\in E(G)$ and $xw'\notin E(G)$.
Note that $x$ is not in $V(H)$ as $w$ and $w'$ are false twins in $H$.
Moreover, it is clear that $x\notin \{u,v\}$ and therefore, $x$ belongs to $(N(u)\cup N(v)) \backslash \{u,v\}$.

By Lemma \ref{lem:reduction} we have that $H$ is a total $(k-2)$-uniform graph.
Thus, $H$ has a total dominating sequence of length $k-2$ starting with $w$, say $(w=w_1,\dots,w_{k-2})$.

Now consider the sequence $S=(w',w_1,\dots,w_{k-2},u,v)$.
Since $x$ is a neighbor of $w_1$ but not of $w'$, we have $N(w_1)\backslash N(w')\neq \emptyset$.
On the other hand,
as $w'$ and $w_1$ are false twins in $H$ and $(w_1,\dots,w_{k-2})$ is a legal sequence in $H$,
we see that $(w',w_1,\dots,w_{k-2})$ is a legal sequence in $G$.
Moreover, as $u\in N(v)$ and $v\in N(u)$ are adjacent to no vertex in $V(H)$, we obtain that $S$ is a legal sequence in $G$.
In addition, it is clear that $\{w',w_1,\dots,w_{k-2},u,v\}$ is a total dominating set of $G$ and hence, $S$ is a total dominating sequence of $G$ of length $k+1$, contradiction.
\end{proof}

Combining the results in Lemmas \ref{lem:reduction} and \ref{lem:falsetwinreduction} yields the following result.

\begin{proposition}\label{prop:kunif-ftf-red}
Let $k\geq 4$ be an even integer.
Let $G\in \mathcal{G}_k$ and $uv$ be an arbitrary edge of $G$.
Then $G\backslash (N[u]\cup N[v]) $ is in $\mathcal{G}_{k-2}$.
\end{proposition}

\section{Total $k$-Uniform Graphs with Small $k$}\label{sec:newtkunif}
In this section, we introduce a new connected total $4$-uniform graph and based on this graph, we present a new connected total 8-uniform graph.

The graph $K_{n,n}-M$, where $M$ is an arbitrary perfect matching of $K_{n,n}$, is called a \emph{crown graph} on $2n$ vertices.
Graphs which are isomorphic to $K_n\cup K_m$, where $m,n\geq 2$, or a crown graph on at least 6 vertices are members of $\mathcal{G}_4$.
It is easy to see that the former ones are the only disconnected graphs in $\mathcal{G}_4$.
The work in \cite{gologranc2019graphs} showed that the latter ones are the only bipartite graphs in $\mathcal{G}_4$ and conjectured that these are the only connected graphs belonging to $\mathcal{G}_4$.
In this section, we disprove their conjecture by giving a counterexample which is the line graph of a complete graph with 6 vertices.

For a given graph $G$, its line graph $L(G)$ is a graph such that 
each vertex of $L(G)$ represents an edge of $G$; and
two vertices of $L(G)$ are adjacent if and only if their corresponding edges are incident.
Particularly, in $L(K_n)$ vertices are pairs of elements in $\{1,2,\dots,n\}$ and two vertices are adjacent when they have a common element.\\

\begin{proposition}\label{prop:total4example}
The graph $L(K_6)$ is total 4-uniform.
\end{proposition}
\begin{proof}
Notice that it suffices to prove that $4\leq \gamma_t(L(K_6))$ and $\gamma_{gr}^t(L(K_6))\leq 4$.
We first show that $\gamma_t(L(K_6))\geq 4$.
On the contrary, suppose that $L(K_6)$ has a total dominating set of size 3, 
say $S=\{v_1,v_2,v_3\}$.
Recall that vertices in $L(K_6)$ correspond to sets of size two.
As $S$ is a total dominating set, one of the members of $S$ is adjacent to both of the other members in $S$.
Without loss of generality, assume that $v_1$ is adjacent to both $v_2$ and $v_3$.
Therefore, $v_1\cap v_2 \neq \emptyset $ and $v_1\cap v_3 \neq \emptyset $, 
and hence,
$|v_1\cup v_2|\leq 3$ and $|v_1\cup v_2 \cup v_3|\leq 4$.
Then a vertex $v_4$ consisting of two elements in $\{1,2,\dots,6\} \backslash (v_1\cup v_2 \cup v_3)$ is adjacent to none of $v_1,v_2$ and $v_3$, and thus we obtain a contradiction.

We next show that $\gamma_{gr}^t(L(K_6))\leq 4$.
Assume to the contrary that $L(K_6)$ has a legal sequence $(v_1,v_2,v_3,v_4,v_5)$.
Then $v_5$ has a neighbor which is adjacent to none of $v_1,v_2,v_3$ and $v_4$, and hence,
we obtain $|\cup_{i=1}^4 v_i |\leq 4$.
As $v_1,v_2,v_3$ and $v_4$ are distinct sets of size two, $\cup_{i=1}^4 v_i$ cannot be of size 3 or less.
Therefore, we see that $|\cup_{i=1}^4 v_i |=4$.
Since $v_4$ has a neighbor which is adjacent to none of $v_1,v_2$ and $v_3$,
we have that $v_4$ has an element not contained in $\cup_{i=1}^3 v_i$.
Consequently, we get $|\cup_{i=1}^3 v_i|=3$, which implies that $v_1=\{a,b\}$, $v_2=\{b,c\}$ and $v_3=\{a,c\}$ for some $a,b,c\in \{1,2,\dots,6\}$.
Then, as $v_3$ is contained in $v_1\cup v_2$, any neighbor of $v_3$ is adjacent to $v_1$ or $v_2$, and thus,
we obtain a contradiction with the assumption of $(v_1,v_2,v_3,v_4,v_5)$ being legal.
\end{proof}
Clearly $L(K_6)$ is connected and false twin-free but not a bipartite graph (as $\{1,2\}, \{2,3\},\{1,3\}$ form a triangle) and therefore, 
it is a counterexample disproving Conjecture 3.2 in \cite{gologranc2019graphs}.
In other words, a connected total 4-uniform graph does not have to be a bipartite graph.

\begin{remark}
Proposition \ref{prop:total4example} also plays a role in constructing a counterexample for the converse of the statement in Proposition \ref{prop:kunif-ftf-red} as follows. Indeed, if $G\backslash (N[u]\cup N[v]) $ is a graph in $\mathcal{G}_{k-2}$ for any edge $uv$ of $G$, then $G$ does not have to be in $\mathcal{G}_k$.
As an example of $G$ consider $L(K_9)$.
For any two adjacent vertices $\{a,b\}$ and $\{b,c\}$ in $L(K_9)$, removing these two vertices together with their neighbors results in a graph isomorphic to $L(K_6)$, which is total 4-uniform.
However, $L(K_9)$ is not a total $k$-uniform graph since $(\{1,2\}, \{2,3\}, \{3,4\},\{5,6\}, \{6,7\},\{7,8\})$ and $(\{1,9\} \{2,9\},\dots, \{7,9\})$ are total dominating sequences in $L(K_9)$ of lengths $6$ and $7$, respectively. 

\end{remark}

We next make use of product of graphs to obtain new total $k$-uniform graphs.
The \emph{direct product} $G\times H$ of graphs $G$ and $H$ is a graph such that the vertex set of $G\times H$ is the Cartesian product $V(G)\times V(H)$ and vertices $(g,h)$ and $(g',h')$ are adjacent in $G\times H$ if and only if $gg'\in E(G)$ and $hh'\in E(H)$.
In particular, for a graph $G$ where $V(G)=\{v_1,v_2,...,v_n\}$, the graph $G\times K_2$ (also called \emph{bipartite double cover} of $G$) is a bipartite graph with parts $\{v_1',v_2',...,v_n'\}$ and $\{v_1'',v_2'',...,v_n''\}$ satisfying $v_i',v_j''\in E(G\times K_2)$ if and only if $v_i,v_j\in E(G)$.
See Figure \ref{fig:g*} for an example of $G\times K_2$.

 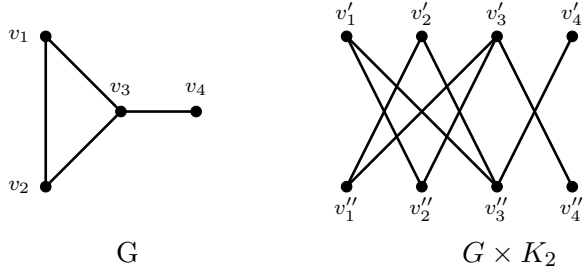
\begin{figure}\centering
 \begin{tikzpicture}[line cap=round,line join=round,>=triangle 45,x=1cm,y=1cm]
 \clip(0.5,-0.1) rectangle (8.5,4);
 \draw [line width=1pt] (1,3)-- (1,1);
 \draw [line width=1pt] (1,1)-- (2,2);
 \draw [line width=1pt] (2,2)-- (1,3);
 \draw [line width=1pt] (2,2)-- (3,2);
 \draw [line width=1pt] (5,3)-- (6,1);
 \draw [line width=1pt] (5,3)-- (7,1);
 \draw [line width=1pt] (6,3)-- (5,1);
 \draw [line width=1pt] (6,3)-- (7,1);
 \draw [line width=1pt] (7,3)-- (5,1);
 \draw [line width=1pt] (7,3)-- (6,1);
 \draw [line width=1pt] (7,3)-- (8,1);
 \draw [line width=1pt] (8,3)-- (7,1);
 \draw (1.8,0.4) node[anchor=north west] {G};
 \draw (6.4,0.4) node[anchor=north west] {$G\times K_2$};
 \begin{scriptsize}
 \draw [fill=black] (1,3) circle (2pt);
 \draw[color=black] (0.65,3) node {$v_1$};
 \draw [fill=black] (1,1) circle (2pt);
 \draw[color=black] (0.65,1) node {$v_2$};
 \draw [fill=black] (2,2) circle (2pt);
 \draw[color=black] (2,2.3) node {$v_3$};
 \draw [fill=black] (3,2) circle (2pt);
 \draw[color=black] (3,2.3) node {$v_4$};
 \draw [fill=black] (5,3) circle (2pt);
 \draw[color=black] (5,3.3) node {$v_1'$};
 \draw [fill=black] (6,3) circle (2pt);
 \draw[color=black] (6,3.3) node {$v_2'$};
 \draw [fill=black] (7,3) circle (2pt);
 \draw[color=black] (7,3.3) node {$v_3'$};
 \draw [fill=black] (8,3) circle (2pt);
 \draw[color=black] (8,3.3) node {$v_4'$};
 \draw [fill=black] (5,1) circle (2pt);
 \draw[color=black] (5,0.7) node {$v_1''$};
 \draw [fill=black] (6,1) circle (2pt);
 \draw[color=black] (6,0.7) node {$v_2''$};
 \draw [fill=black] (7,1) circle (2pt);
 \draw[color=black] (7,0.7) node {$v_3''$};
 \draw [fill=black] (8,1) circle (2pt);
 \draw[color=black] (8,0.7) node {$v_4''$};
 \end{scriptsize}
 \end{tikzpicture}
 \caption{A graph $G$ and its bipartite double cover $G\times K_2$.}\label{fig:g*}
 \end{figure}

\begin{theorem}[Theorem 5.29 in \cite{imklav}] \label{thm:conGH}
Let $G$ and $H$ be graphs with at least one edge.
Then $G\times H$ is connected if and only if $G$ and $H$ are connected and at least one of them is non-bipartite.
Furthermore, if both $G$ and $H$ are connected and bipartite,
then $G\times H$ has exactly two connected components.
\end{theorem}

Therefore, by Theorem \ref{thm:conGH} we observe that if $G$ is a bipartite graph, then the graph $G\times K_2$ has two connected components, and it is easy to see that both are isomorphic to $G$.
On the other hand, $G\times K_2$ is a connected graph when $G$ is non-bipartite.
Now, for a connected, non-bipartite and total $k$-uniform graph $G$, consider a total dominating sequence of $G\times K_2$.
It is easy to verify that a subsequence of the vertices belonging to the same part corresponds to a total dominating sequence of $G$.
Thus, since $G$ is a total $k$-uniform graph, we conclude that
any total dominating sequence of $G\times K_2$ has $k$ elements from both parts and hence, it is a total $2k$-uniform graph.
Consequently, we obtain the following result which allows to create a connected total $2k$-uniform graph based on a connected non-bipartite total $k$-uniform graph. 

\begin{theorem}\label{thm:bipartite}
If the graph $G$ is connected, non-bipartite and total $k$-uniform, 
then $G\times K_2$ is a connected total $2k$-uniform graph.
\end{theorem}

\begin{remark}
Note that $K_n$ is a connected, non-bipartite and total 2-uniform graph when $n\geq 3$, and a crown graph on $2n$ vertices is indeed $K_n \times K_2$.
Therefore, the fact that a crown graph on at least 6 vertices being a connected total 4-uniform graph is a simple application of Theorem \ref{thm:bipartite}.
\end{remark}

Combining the results in Proposition \ref{prop:total4example} and Theorem \ref{thm:bipartite}, we obtain a connected total 8-uniform graph which does not appear in the literature.
\begin{corollary} 
The graph $L(K_6)\times K_2$ is connected and total 8-uniform.
\end{corollary}

\section{Regularity}\label{sec:reg}
In this section, for every even positive integer $k$ we show that any connected graph in $\mathcal{G}_k$ is regular.
We start with studying bipartite graphs.

\begin{proposition}\label{prop:bipartreg}
Let $G$ be a bipartite, connected, false twin-free and total $k$-uniform graph for some even positive integer $k$.
Then, $G$ is a regular graph.
\end{proposition}
\begin{proof}
The proof is by strong induction on $k$.
For $k=2$, such a graph is $K_2$ and trivially it is regular.
Now let $k\geq 4$ be an even integer and assume that the statement is true for every positive even $k'$ less than $k$.

Let $X$ and $Y$ be the parts of $G$.
Let $u$ and $v$ be two vertices of $X$ having a common neighbor $w\in Y$.
Since $G\in \mathcal{G}_k$, Proposition \ref{prop:kunif-ftf-red} implies that the graph $H=G\backslash (N[u]\cup N[w]) $ is in $\mathcal{G}_{k-2}$.
Let $H_1,\dots,H_r$ be connected components of $H$.
Then, each $H_i$ is bipartite, false twin-free and total $k_i$-uniform for some even positive integer $k_i\leq k-2<k$.
Therefore, the induction hypothesis implies that every $H_i$ is regular.
It is clear that parts of a bipartite regular graph are of the same size and hence, we see that the parts of $H$ are of the same size as well.
Consequently, we obtain $|X|-|N(w)|=|Y|-|N(u)|$.
Similarly, by considering $G\backslash (N[v]\cup N[w]) $ we get $|X|-|N(w)|=|Y|-|N(v)|$ and thus, we have $|N(u)|=|N(v)|$.
In other words, any two vertices in $X$ sharing a common neighbor are of the same degree.

Now, let $u$ and $v$ be any two vertices in $X$.
Since $G$ is connected, there exists a path $u=x_1,y_1,x_2,y_2,\dots, x_{s-1},y_{s-1},x_s=v$ between $u$ and $v$, where $x_1,\dots,x_{s}\in X$ and $y_1,\dots,y_{s-1} \in Y$.
The result above yields that $x_i$ and $x_{i+1}$ have the same degree for each $i=1,\dots,s-1$ and thus, we obtain that any two vertices in $X$ are of the same degree, that is,
there exists a positive integer $a$ such that $|N(x)|=a$ for every $x\in X$. 
In a similar manner, we see that $|N(y)|=b$ for every $y\in Y$ for some positive integer $b$.

A double counting argument on the edges of $G$ gives $|X|\cdot a=|Y|\cdot b$.
Let $d$ be the greatest common divisor of $a$ and $b$.
Then, $a=da_1$ and $b=db_1$ where $a_1$ and $b_1$ are relatively prime.
Therefore, there exists a positive integer $c$ such that $|X|=b_1c$ and $|Y|=a_1c$.
On the other hand, recall that we have $|X|-|N(w)|=|Y|-|N(u)|$ where $u\in X$, $w\in Y$ and $uw\in E(G)$, which gives $b_1c-db_1=a_1c-da_1$.
Then, we have $(c-d)(b_1-a_1)=0$, that is, $c=d$ or $a_1=b_1$.
As $k\neq 2$, $|X|>b$ and hence, $c>d$.
Therefore, we get $a_1=b_1$ which yields that $a=b$ and thus, $G$ is regular.
\end{proof}

\begin{theorem}
For every positive even integer $k$, a connected, false twin-free and total $k$-uniform graph is regular.
\end{theorem}
\begin{proof}
Let $G$ be a connected graph in $\mathcal{G}_k$.
If $G$ is a bipartite graph, then Proposition \ref{prop:bipartreg} implies that $G$ is regular.
Suppose that $G$ is non-bipartite.
Then Theorem \ref{thm:bipartite} yields that the graph $G\times K_2$ is connected and total $2k$-uniform.
Recall that by construction $G\times K_2$ is bipartite.
Obviously, two vertices from different parts can not be false twins.
Moreover, any two vertices belonging to the same part are not false twins since $G$ is false twin-free and hence, we obtain that $G\times K_2$ is false twin-free as well.
Then we can apply Proposition \ref{prop:bipartreg} on $G\times K_2$ and observe that it is regular which implies that $G$ is regular.
\end{proof}

\section{Total $k$-Uniform Chordal Graphs}\label{sec:chordal}

In this section, we characterize all total $k$-uniform chordal graphs.
In \cite{gologranc2019graphs} Theorem 3.3 states that there is no connected chordal graph $G$ with $\gamma_t(G)=\gamma_{gr}^t(G)=4$.
We extend their result and show that there is no connected total $k$-uniform chordal graph when $k\geq 4$.
Notice that to show this result, it suffices to prove that $\mathcal{G}_k$ has no connected chordal graph for $k\geq 4$.

\begin{lemma}\label{lem:c5c6}
If $G$ is a connected graph in $\mathcal{G}_k$ where $k\geq 4$,
then $G$ has an induced $C_5$ or $C_6$.
\end{lemma}
\begin{proof}
Let $G$ be a connected graph in $\mathcal{G}_k$ with $k\geq 4$.
It is well-known that a graph is $P_3$-free if and only if it is disjoint union of complete graphs.
Then, since $k\geq4$, we see that $G$ has an induced $P_3$.
Let $u,v,w$ be an induced path in $G$ (that is, $uv,vw\in E(G)$ but $uw\notin E(G)$).
We show that $H=G\backslash (N[u]\cup N[v]\cup N[w])$ has an isolated vertex.
Suppose that $H$ has no isolated vertices.
Note that $H$ is not the empty graph since $k\geq 4$. 
Then, by Corollary \ref{cor:eventotaldomseq} we see that $H$ has a total dominating sequence $(v_1,\dots,v_m)$ such that $m$ is even.
As $G$ is false twin-free, $N(u)\neq N(w)$.
Without loss of generality assume that $w$ has a neighbor which is not adjacent to $u$.
Then, it is easy to verify that $(u,w,v,v_1,\dots,v_m)$ is a total dominating sequence of $G$ and thus , we get $m+3=k$.
However, $m+3$ is odd and we get a contradiction.
Consequently, there exists an isolated vertex $z$ in $H$.

Lemma \ref{lem:kunifiso} implies that $z$ is not an isolated vertex in $G\backslash (N[v]\cup N[w])$ and hence, $z$ has a neighbor $x$ such that $x$ is adjacent to only $u$ among $u,v$ and $w$.
Similarly, $G\backslash (N[u]\cup N[v])$ has no isolated vertex and therefore, there exists a neighbor $y$ of $z$ which is adjacent to 
only $w$ among $u,v$ and $w$.
Finally, if $x$ and $y$ are not adjacent, then the subgraph of $G$ induced by $\{u,v,w,x,y,z\}$ is a $C_6$; and if $x$ and $y$ are adjacent, 
then the subgraph of $G$ induced by $\{u,v,w,x,y,\}$ is a $C_5$.
\end{proof}

Notice that removing a false twin from a connected chordal graph does not affect the connectedness or being chordal.  
Then, since any total $k$-uniform graph has a subgraph in $\mathcal{G}_k$, we obtain the following conclusion by Lemma \ref{lem:c5c6}:
\begin{theorem}\label{thm:nochordal}
For any $k\geq 4$, there does not exist a total $k$-uniform connected chordal graph.
\end{theorem}

Theorem \ref{thm:nochordal} implies that a total $k$-uniform chordal graph must be union of $k/2$ number of total 2-uniform chordal graphs.
On the other hand, it is clear that a complete multipartite graph is chordal if and only if at most one partite has 2 or more vertices.
We hereby obtain a complete characterization of total $k$-uniform chordal graphs:
\begin{theorem}\label{thm:chordals}
Let $G$ be a graph without isolated vertices.
$G$ is a total $k$-uniform chordal graph if and only if $G$ is disjoint union of complete multipartite graphs in which at most one partite is of size greater than 1.
\end{theorem}

The \emph{girth} of a graph $G$, denoted by $g(G)$, is the length of a shortest cycle (if any) in $G$.
Acyclic graphs (forests) are considered to have infinite girth.
A graph isomorphic to $K_{1,n}$ for some $n$ is called a \emph{star}.
Clearly, a complete multipartite graph is a tree only if it is a star.
As a result, Lemma \ref{lem:c5c6} yields the following result on the girth of a total $k$-uniform graph:

\begin{theorem}
If $G$ is a total $k$-uniform graph, then either $G$ is disjoint union of $k/2$ stars or $g(G)\leq 6$.
\end{theorem}

\section{Discussion and Conclusions}\label{sec:dis}
For a fixed positive integer $k$, 
the problem of determining whether a given graph is total $k$-uniform is clearly solvable in polynomial time.
On the other hand, both of the problems of finding the total domination number and finding the Grundy total domination number of a given graph are NP-complete even in bipartite graphs (see, \cite{pfaff1983np} and \cite{brevsar2016total}, respectively).
One research direction is to solve the decision problem of determining whether a given graph is total $k$-uniform for some $k$.
In this paper, we partially answered this problem by Theorem \ref{thm:chordals}, which characterizes all total $k$-uniform chordal graphs.

In Section \ref{sec:newtkunif}, we presented a non-bipartite, connected and total 4-uniform graph, which is a counterexample for a conjecture by \cite{gologranc2019graphs}.
Yet, we have been unable to characterize all total 4-uniform graphs and therefore, 
another potential research direction is to complete this task.

The connection between total 6-uniform graphs and the existence of finite affine planes provided in \cite{gologranc2019graphs} is interesting and shows that finding connected total $k$-uniform graphs becomes more complicated as $k$ gets larger.
In Figure 1 of \cite{gologranc2019graphs} a connected total 6-uniform graph is presented.
In Section \ref{sec:newtkunif}, we provided a connected total 8-uniform graph.
We believe that there exist connected total $k$-uniform graphs also for every even $k\geq 10$ and
finding such graphs is a topic of ongoing research.

In a total $k$-uniform graph, any legal sequence can be extended to a total dominating sequence irrespective of what the initial vertex is. 
Therefore, one would naturally expect a symmetrical structure in the graph.
As one of the main results of this paper, we proved that every connected, false twin-free and total $k$-uniform graph is regular, which is a strong structural result for total $k$-uniform graphs and might help one to solve some of the research problems we posed.

\cite{erey2020uniform} presents another version of total $k$-uniform graphs.
A sequence $(v_1,\dots,v_k)$ is called a \emph{dominating open neighborhood sequence} if $\{v_1,\dots,v_k\}$ is a dominating set and $N(v_i)\backslash \bigcup_{j=1}^{i-1} N(v_j)\neq \emptyset$
holds for every $i \in \{2, \dots , k\}$.
A graph is called \emph{open $k$-uniform} if every dominating open neighborhood sequence has length $k$.
It is easy to see that (also remarked in \cite{erey2020uniform}) every open $k$-uniform graph is also total $k$-uniform and hence, the family of open $k$-uniform graphs is a subclass of total $k$-uniform graphs.
Therefore, all the results we obtain in this paper are valid for open $k$-uniform graphs as well.
For example, there is no open $k$-uniform graph when $k$ is odd.
We believe that the methods we use in this paper can lead to a complete characterization of open $k$-uniform graphs.

\section*{Acknowledgments}
This work is supported by the Scientific and Technological Research
Council of Turkey (TUBITAK) under grant no. 118E799. The work of Didem G\"{o}z\"{u}pek is also supported by the BAGEP Award of the Science Academy of Turkey.
We would like to thank Douglas Rall for making us aware of some properties of the direct product of graphs in Section \ref{sec:newtkunif}. 

\end{document}